\documentclass[11pt]{article}

\usepackage[margin=1in]{geometry}
\usepackage[T1]{fontenc}
\usepackage{lmodern}
\usepackage{amsmath,amssymb,amsthm,mathtools}
\usepackage{booktabs}
\usepackage{enumitem}
\usepackage{microtype}
\usepackage{xcolor}
\usepackage{hyperref}
\usepackage{array}
\usepackage{longtable}
\usepackage{mathrsfs}
\usepackage{aliascnt}
\usepackage[nameinlink,capitalize]{cleveref}

\hypersetup{
  colorlinks=true,
  linkcolor=blue!55!black,
  citecolor=green!45!black,
  urlcolor=blue!60!black
}

\newtheorem{theorem}{Theorem}[section]
\newaliascnt{lemma}{theorem}
\newtheorem{lemma}[lemma]{Lemma}
\aliascntresetthe{lemma}
\newaliascnt{corollary}{theorem}
\newtheorem{corollary}[corollary]{Corollary}
\aliascntresetthe{corollary}
\newaliascnt{proposition}{theorem}
\newtheorem{proposition}[proposition]{Proposition}
\aliascntresetthe{proposition}
\newaliascnt{claim}{theorem}

\aliascntresetthe{claim}
\newaliascnt{definition}{theorem}
\newtheorem{definition}[definition]{Definition}
\aliascntresetthe{definition}
\newaliascnt{remark}{theorem}

\aliascntresetthe{remark}
\newaliascnt{assumption}{theorem}

\aliascntresetthe{assumption}

\crefname{theorem}{Theorem}{Theorems}
\Crefname{theorem}{Theorem}{Theorems}
\crefname{lemma}{Lemma}{Lemmas}
\Crefname{lemma}{Lemma}{Lemmas}
\crefname{corollary}{Corollary}{Corollaries}
\Crefname{corollary}{Corollary}{Corollaries}
\crefname{proposition}{Proposition}{Propositions}
\Crefname{proposition}{Proposition}{Propositions}

\newcommand{\R}{\mathbb{R}}

\newcommand{\E}{\mathbb{E}}
\newcommand{\Prob}{\mathbb{P}}
\newcommand{\one}{\mathbf{1}}
\newcommand{\eps}{\varepsilon}
\newcommand{\cF}{\mathcal{F}}
\newcommand{\cG}{\mathcal{G}}
\newcommand{\cH}{\mathcal{H}}
\newcommand{\cU}{\mathcal{U}}
\newcommand{\cV}{\mathcal{V}}
\newcommand{\cX}{\mathcal{X}}
\newcommand{\cY}{\mathcal{Y}}

\newcommand{\Span}{\operatorname{span}}
\newcommand{\rank}{\operatorname{rank}}

\newcommand{\op}{\mathrm{op}}
\newcommand{\Gap}{\operatorname{Gap}}
\newcommand{\Gin}{\operatorname{Gin}}

\newcommand{\dd}{\,\mathrm{d}}

\newcommand{\authornames}{Wendao Wu, Cong Fang$^{\dagger}$}
\newcommand{\affiliation}{Peking University}
\newcommand{\contactemail}{colombo.2026@outlook.com}

\newcommand{\ResearchAgentSystem}{Colombo, our laboratory's internal auto-research system}

\title{\textbf{Randomized Matvec Lower Bounds for Simplex-Based Matrix Games}}
\author{\authornames\\[0.4em]
  \normalsize\affiliation\\[0.2em]
  \small\href{mailto:\contactemail}{\texttt{\contactemail}}}
\date{October 2026}
\hypersetup{pdftitle={Randomized Matvec Lower Bounds for Simplex-Based Matrix Games},pdfauthor={\authornames}}

\begin{document}
\maketitle

\begin{abstract}
We present a proof of randomized matrix-vector lower bounds for two normalized matrix-game geometries: a Euclidean ball against a simplex with row norms at most one, and two simplices with entries of absolute value at most one.  A query returns $(Ax,A^\top y)$ for arbitrary real vectors, and the algorithm must return a feasible pair whose full saddle-point gap is at most $\eps$, with probability at least $2/3$ on every admissible matrix.  Let $Q_{\rm rand}^{2,1}(m,n,\eps)$ and $Q_{\rm rand}^{1,1}(m,n,\eps)$ denote the corresponding worst-case query complexities.  For sufficiently small $\eps$, we obtain polynomial-dimensional hard instances with
\[
\begin{aligned}
 Q_{\rm rand}^{2,1}(m,n,\eps)
 &\ge \Omega\!\left(\frac{\eps^{-2/3}}{\log^2(1/\eps)\log\log(1/\eps)}\right),\\
 Q_{\rm rand}^{1,1}(m,n,\eps)
 &\ge \Omega\!\left(\frac{\eps^{-2/3}}{\log^{7/3}(1/\eps)\log\log(1/\eps)}\right).
\end{aligned}
\]
The hard dimensions have orders $\eps^{-2/3}$ and $\eps^{-2/3}/\log^{1/3}(1/\eps)$, respectively, and the bounds extend to larger dimensions.  Their accuracy exponents match the deterministic upper bounds of Karmarkar, O'Carroll, and Sidford up to logarithms \cite{KOS26}.  The proof extracts a fresh Gaussian core after adaptive two-sided queries and uses its smallest-singular-value uncertainty to force linear-system solve hardness.  Two residual-to-gap reductions transfer this hardness to the game geometries, with an additional logarithmic normalization loss only for simplex--simplex games.
\end{abstract}


\begin{samepage}
\noindent\textbf{AI Usage.}
Nearly the entire research pipeline for this paper was carried out by
\ResearchAgentSystem{}, powered by GPT-5.6 Sol. The system also conducted a
Lean-backed article audit of the resulting manuscript. The authors subsequently reviewed and approved the
mathematical claims, presentation, and formal artifacts, and take
responsibility for the final manuscript. The complete Lean audit report and the system's technical report will be made public at a later date.
\par
\end{samepage}

\paragraph{Publication Intent and Recognition of Independent Work.}

This preprint is released as a research output of our study of
automated scientific research, for public scrutiny and discussion.
As we do not intend to submit this manuscript to a journal or
conference, it has not undergone the full editorial refinement
and internal review we would normally require for submission.
The authors remain responsible for its contents and for correcting
any errors or inaccurate attribution.

We recognize that releasing these results may affect researchers
who have independently been pursuing the same questions, including
work that is not yet public.
If this applies to your work, we invite you to contact us at
\texttt{colombo.2026@outlook.com}.
A brief description of the overlap, together with any dated
materials you are comfortable sharing, would help us understand
its scope and chronology.
Such materials might include draft excerpts, presentation slides,
or revision records; a complete manuscript or unpublished proof
is not required to begin the conversation.

We will review these cases carefully and update citations,
acknowledgments, and accounts of independent or concurrent work
as appropriate, including revising priority claims where warranted.
Nonpublic materials will be treated confidentially and used only
to assess the overlap unless their authors authorize otherwise.
We will obtain explicit permission from the authors before
publicly discussing such materials or providing them to an
automated research system.

\section{Introduction}\label{sec:introduction}
We study the number of matrix-vector products needed to compute an approximate equilibrium of
\[
  \min_{x\in\cX}\max_{y\in\cY}y^\top Ax,
  \qquad A\in\R^{m\times n}.
\]
Here $\cY=\Delta_m$ is the probability simplex, and $\cX$ is either the Euclidean unit ball $\mathbb B_2^n$ or the probability simplex $\Delta_n$.  These are the $\ell_2$--$\ell_1$ and $\ell_1$--$\ell_1$ geometries, respectively.  In the first geometry every row of $A$ has Euclidean norm at most one; in the second, every entry has absolute value at most one.  For an accuracy $\eps>0$, the algorithm must return a feasible pair $(x,y)$ with full saddle-point gap
\begin{equation}\label{eq:intro-target}
  \max_{y'\in\cY}(y')^\top Ax
  -\min_{x'\in\cX}y^\top Ax'
  \le\eps.
\end{equation}
One query returns $(Ax,A^\top y)$.  We count these two-sided queries, allow arbitrary real query vectors and adaptive randomization, and require success probability at least $2/3$ on every admissible matrix.  The returned strategies must belong to their domains, but need not have been queried.

\paragraph{The benchmark and the randomized question.}
Karmarkar, O'Carroll, and Sidford give deterministic algorithms using $\widetilde O(\eps^{-2/3})$ queries in both geometries \cite{KOS26}.  Their earlier bounds were $\widetilde O(\eps^{-7/9})$ for ball--simplex games and $\widetilde O(\eps^{-8/9})$ for simplex--simplex games \cite{KOS25}.  Kornowski and Shamir establish the corresponding $\widetilde\Omega(\eps^{-2/3})$ lower bounds for deterministic algorithms in sufficiently large dimensions, and identify randomized complexity as an open question \cite{KS26}.  Thus the remaining issue is whether randomization can improve the achievable accuracy exponent.  A deterministic resisting oracle does not settle this issue: the hard matrix must be fixed independently of a randomized algorithm's realized private seed.

\paragraph{The answer developed here.}
This proof gives randomized worst-case lower bounds with the same accuracy exponent.  Write $Q_{\rm rand}^{p,1}(m,n,\eps)$ for the minimum uniform query cap needed to satisfy \eqref{eq:intro-target} in geometry $p\in\{1,2\}$.  For all sufficiently small $\eps$, \Cref{cor:rates} gives
\begin{align*}
 Q_{\rm rand}^{2,1}(m,n,\eps)
 &\ge \Omega\!\left(\frac{\eps^{-2/3}}{\log^2(1/\eps)\log\log(1/\eps)}\right),\\
 Q_{\rm rand}^{1,1}(m,n,\eps)
 &\ge \Omega\!\left(\frac{\eps^{-2/3}}{\log^{7/3}(1/\eps)\log\log(1/\eps)}\right).
\end{align*}
The ball--simplex construction has $m=4N$ and $n=2N+1$ with $N=\Theta(\eps^{-2/3})$.  The simplex--simplex construction has $m=n=4N$ with $N=\Theta(\eps^{-2/3}/\log^{1/3}(1/\eps))$.  Each bound extends to all larger dimensions by \Cref{lem:dimension-extension}.  Consequently, randomization cannot improve the $2/3$ exponent in these regimes if the proof is validated; the upper bounds and their exponent are inherited from \cite{KOS26}.  We do not claim matching logarithms or a joint characterization of dimension and accuracy.

\begin{table}[t]
\centering
\small
\setlength{\tabcolsep}{4pt}
\setlength{\belowcaptionskip}{8pt}
\caption{Two-sided matvec complexity for full saddle-point gap at most $\eps$ under the stated matrix normalizations.  Tildes suppress logarithms in dimension and inverse accuracy.  In the prior lower-bound row, $T$ is the deterministic query budget and the listed dimension conditions are sufficient.  The present hard-core parameter $N$ is related to accuracy above; exact sizes and logarithmic losses appear in \Cref{thm:main-sequences} and \Cref{cor:rates}.}\label{tab:comparison}
\begin{tabular}{@{}>{\raggedright\arraybackslash}p{.22\textwidth}>{\raggedright\arraybackslash}p{.10\textwidth}>{\raggedright\arraybackslash}p{.16\textwidth}>{\raggedright\arraybackslash}p{.16\textwidth}>{\raggedright\arraybackslash}p{.25\textwidth}@{}}
\toprule
Work & Geometry & Upper bound & Lower bound & Algorithms and regime\\
\midrule
Karmarkar et al.\ \cite{KOS25} & $\ell_2$--$\ell_1$ & $\widetilde O(\eps^{-7/9})$ & --- & Deterministic; all dimensions\\[3pt]
 & $\ell_1$--$\ell_1$ & $\widetilde O(\eps^{-8/9})$ & --- & Deterministic; all dimensions\\[3pt]
Karmarkar et al.\ \cite{KOS26} & Both & $\widetilde O(\eps^{-2/3})$ & --- & Deterministic; all dimensions\\[3pt]
Kornowski--Shamir \cite{KS26} & Both & --- & $\widetilde\Omega(\eps^{-2/3})$ & Deterministic; $m\gtrsim T^2\log T$, $n\gtrsim T$\\
\midrule
\textbf{This Paper}\newline (\Cref{cor:rates}) & $\ell_2$--$\ell_1$ & --- & $\boldsymbol{\widetilde\Omega(\eps^{-2/3})}$ & Randomized; $m\ge4N$, $n\ge2N+1$\\[3pt]
\textbf{This Paper}\newline (\Cref{cor:rates}) & $\ell_1$--$\ell_1$ & --- & $\boldsymbol{\widetilde\Omega(\eps^{-2/3})}$ & Randomized; $m,n\ge4N$\\
\bottomrule
\end{tabular}
\end{table}

\paragraph{Why Gaussian uncertainty survives adaptivity.}
The proof begins with an information obstacle rather than a restricted iteration rule.  After adaptive queries to a square Gaussian matrix $G$, the unobserved part remains Gaussian in the orthogonal complements of the queried subspaces.  This posterior identity alone is insufficient: the observed blocks still affect the smallest singular value of the whole matrix.  \Cref{lem:gaussian-core} resolves that obstruction by appending queries after the estimator has returned its output.  They create an invertible square pivot.  The resulting Schur complement is a fresh Gaussian matrix plus a low-rank shift; a kernel restriction removes the shift while preserving a fresh square Gaussian core.  The smallest singular value of $G$ is bounded above by that of the core.  A uniform hard-edge estimate then forces a positive conditional probability of a small singular value for almost every transcript, ruling out highly reliable constant-factor estimation in \Cref{thm:spectral-estimation}.

\paragraph{From spectral information to the game gap.}
Fresh-block arguments and reductions between spectral estimation and solving have a close predecessor in the Wishart-based lower bounds of Braverman, Hazan, Simchowitz, and Woodworth \cite{BHSW20}.  Here the spectral argument treats direct access to both $Gv$ and $G^\top u$, which is not the same information as multiplication by the positive semidefinite matrix $G^\top G$.  A hypothetical fast constant-residual solver for a symmetric indefinite dilation of $G$ would implement approximate inverse power.  Verified refinement controls errors even when successive right-hand sides are adaptive, giving \Cref{thm:solve-lb}.  The two game reductions then turn a small full gap into a small linear-system residual.  For ball--simplex games, a homogeneous coordinate must stay positive at every sufficiently accurate output, allowing safe recovery of the solution.  For simplex--simplex games, a signed orthogonal dictionary detects the Euclidean residual in every direction; random rotation controls its entries at the cost of a $\sqrt{\log N}$ normalization factor.  These steps prove the two parts of \Cref{thm:main-sequences}.

\paragraph{Organization.}
\Cref{sec:related} places the argument among first-order game methods and randomized linear-algebra lower bounds.  \Cref{sec:model} gives the precise models and main statements.  The probabilistic inputs, adaptive spectral argument, and source solve lower bound appear in \Cref{sec:inputs,sec:gaussian-core,sec:solve}.  \Cref{sec:embedding} proves the game reductions, and \Cref{sec:scope} discusses the scope of the result.

\section{Related Work}\label{sec:related}
\paragraph{First-order methods and matrix-vector complexity.}
Nemirovski's prox method for monotone variational inequalities gives an $O(1/T)$ convergence framework for smooth convex--concave saddle problems \cite{Nemirovski04}.  Nesterov's smoothing method provides another route to the same inverse-accuracy order for structured nonsmooth problems, including max representations arising in matrix games \cite{Nesterov05}.  These methods explain the classical $\widetilde O(\eps^{-1})$ benchmark under the normalizations considered here.  Karmarkar, O'Carroll, and Sidford improve the matvec count below that order by exploiting the linear structure of games \cite{KOS25,KOS26}.  Their near-optimal result combines a prox multi-point outer loop with maintained matrix approximations and an amortized analysis of the queries used by regularized subproblem solvers \cite{KOS26}.  Our contribution concerns the necessity of the resulting accuracy exponent; it does not supply a new game algorithm or improve those upper bounds.

\paragraph{Deterministic and randomized game lower bounds.}
Kornowski and Shamir analyze simplex-based games under first-order information, distinguishing one-sided from two-sided matrix-vector access \cite{KS26}.  Their deterministic lower bounds use resisting constructions and cover a broader family of $\ell_p$--simplex geometries.  Their discussion of randomization identifies a specific obstruction: the construction depends on simulating the algorithm's responses, while known nonlinear modifications for randomized lower bounds do not preserve a bilinear objective.  We address the two endpoint geometries $p=1,2$ through Gaussian posterior uncertainty and a full-gap-to-residual implication.  The hard source matrix is drawn independently of the algorithm's seed, and the final contradiction yields a fixed worst-case game instance for each randomized algorithm.  A universal distribution over games is a different conclusion and is not constructed here.

\paragraph{Adaptive linear algebra and the Gaussian hard edge.}
Braverman et al.\ prove randomized gradient-query lower bounds for linear regression using spectral estimation of Wishart matrices \cite{BHSW20}.  Their posterior decomposition and surviving Wishart block show why spectral uncertainty can persist beyond Krylov-restricted algorithms.  Our Gaussian posterior calculation follows the same deferred-information principle, but direct two-sided access to a nonsymmetric Gaussian matrix requires the square-pivot and low-rank-shift argument in \Cref{lem:gaussian-core}.  We import the finite-dimensional smallest-singular-value estimate of Huang and Tikhomirov \cite{HuangTikhomirov}, specialized to the first power of a Gaussian matrix, and a standard Gaussian operator-norm tail \cite{Vershynin}.  These distributional facts supply the separated spectral events and the bounded-condition-number source class; they do not themselves establish adaptive query hardness.

\section{Models and Main Results}\label{sec:model}

\subsection{Two matrix-game models}
For $k\ge1$, let
\[
  \Delta_k:=\{z\in\R^k_{\ge0}:\one^\top z=1\},
  \qquad
  \mathbb B_2^k:=\{z\in\R^k:\|z\|_2\le1\}.
\]
For compact convex sets $\cX\subseteq\R^n$, $\cY\subseteq\R^m$, a payoff matrix $A\in\R^{m\times n}$, and $(x,y)\in\cX\times\cY$, define
\begin{equation}\label{eq:general-gap}
  \Gap_A^{\cX,\cY}(x,y)
  :=\max_{y'\in\cY}(y')^\top Ax
    -\min_{x'\in\cX}y^\top Ax'.
\end{equation}
A two-sided matvec query at $(x,y)$ returns
\[
  \mathcal O_A(x,y):=(Ax,A^\top y).
\]
The lower bounds allow arbitrary real queries and require feasible outputs.  They also cover algorithms that query only within the strategy domains.

\paragraph{$\ell_1$--$\ell_1$ games.}
Here $\cX=\Delta_n$, $\cY=\Delta_m$, and $\|A\|_{\max}\le1$.  The full gap is
\begin{equation}\label{eq:l11-gap}
  \Gap_A^{1,1}(x,y)
  =\max_{i\in[m]}(Ax)_i-\min_{j\in[n]}(A^\top y)_j.
\end{equation}

\paragraph{$\ell_2$--$\ell_1$ games.}
Here $\cX=\mathbb B_2^n$, $\cY=\Delta_m$, and
\[
  \|A\|_{2\to\infty}
  :=\max_{i\in[m]}\|A_{i,:}\|_2\le1.
\]
The full gap is
\begin{equation}\label{eq:l21-gap}
  \Gap_A^{2,1}(x,y)
  =\max_{i\in[m]}(Ax)_i+\|A^\top y\|_2.
\end{equation}

\begin{definition}[Randomized worst-case query complexity]
Fix success probability $2/3$.  Let $Q_{\rm rand}^{1,1}(m,n,\eps)$ and $Q_{\rm rand}^{2,1}(m,n,\eps)$ denote the minimum numbers of two-sided queries required by randomized adaptive algorithms that solve, respectively, every admissible $\ell_1$--$\ell_1$ or $\ell_2$--$\ell_1$ instance of dimensions $(m,n)$ to full gap at most $\eps$ with probability at least $2/3$.
\end{definition}

We write $\log$ for the natural logarithm; asymptotic constants are universal.  The notation $\widetilde O$ and $\widetilde\Omega$ suppresses only logarithms in dimension and inverse accuracy.  Throughout, ``uses at most $T$ queries'' means an almost-sure hard cap on every oracle transcript, including transcripts outside a promised input class.  This convention loses no generality: a promise algorithm can always be stopped after $T$ queries and made to return a fixed default output.  Correctness is required only on admissible or promised inputs.

\subsection{Main sequence theorem}

\begin{theorem}[Randomized lower bounds in both matrix-game geometries]\label{thm:main-sequences}
For universal constants $c_0,c_1,c_2>0$ and every sufficiently large integer $N$, the following statements hold.
\begin{enumerate}[label=(\alph*)]
  \item \textbf{$\ell_2$--$\ell_1$.}  Let
  \[
    m_N^{(2,1)}=4N,\qquad
    n_N^{(2,1)}=2N+1,\qquad
    \eps_N^{(2,1)}=\frac{c_0}{N^{3/2}}.
  \]
  Every randomized algorithm that returns full gap at most $\eps_N^{(2,1)}$ with probability at least $2/3$ on every admissible $\ell_2$--$\ell_1$ matrix of these dimensions must use at least
  \[
    c_2\frac{N}{\log^2(N+2)\,\log\log(N+e^e)}
  \]
  two-sided matvec queries.

  \item \textbf{$\ell_1$--$\ell_1$.}  Let
  \[
    m_N^{(1,1)}=n_N^{(1,1)}=4N,\qquad
    \eps_N^{(1,1)}=\frac{c_1}{N^{3/2}\sqrt{\log(N+2)}}.
  \]
  Every randomized algorithm that returns full gap at most $\eps_N^{(1,1)}$ with probability at least $2/3$ on every admissible $\ell_1$--$\ell_1$ matrix of these dimensions must satisfy the same lower bound on two-sided queries.
\end{enumerate}
\end{theorem}

\begin{corollary}[High-dimensional randomized exponents]\label{cor:rates}
For all sufficiently small $\eps$, choose the corresponding hard-core size $N$ from \Cref{thm:main-sequences}.  In the $\ell_2$--$\ell_1$ case, $N=\Theta(\eps^{-2/3})$ and the bound holds for $m\ge4N$, $n\ge2N+1$.  In the $\ell_1$--$\ell_1$ case, $N=\Theta(\eps^{-2/3}/\log^{1/3}(1/\eps))$ and it holds for $m,n\ge4N$.  In these respective regimes,
\begin{align}
  Q_{\rm rand}^{2,1}(m,n,\eps)
  &\ge
  \Omega\!\left(
    \frac{\eps^{-2/3}}
    {\log^2(1/\eps)\,\log\log(1/\eps)}
  \right),\label{eq:l21-rate}\\
  Q_{\rm rand}^{1,1}(m,n,\eps)
  &\ge
  \Omega\!\left(
    \frac{\eps^{-2/3}}
    {\log^{7/3}(1/\eps)\,\log\log(1/\eps)}
  \right).\label{eq:l11-rate}
\end{align}
\end{corollary}

\begin{lemma}[Extension from the hard core to larger dimensions]\label{lem:dimension-extension}
A hard $\ell_1$--$\ell_1$ instance of dimensions $(m_0,n_0)$ yields an equally hard instance for every $m\ge m_0$, $n\ge n_0$ by duplicating rows and columns.  A hard $\ell_2$--$\ell_1$ instance of dimensions $(m_0,n_0)$ yields an equally hard instance for every $m\ge m_0$, $n\ge n_0$ by duplicating rows and padding with zero columns.
\end{lemma}
\begin{proof}
For $\ell_1$--$\ell_1$, choose surjections $\rho:[m]\to[m_0]$ and $\chi:[n]\to[n_0]$ and set $A'_{ij}=A_{\rho(i),\chi(j)}$.  Any query to $A'$ is simulated by aggregating coordinates along the fibers of $\chi$ and $\rho$, making one query to $A$, and duplicating the response.  Aggregating any simplex output along the same fibers preserves both terms in \eqref{eq:l11-gap}, hence preserves the full gap.

For $\ell_2$--$\ell_1$, first duplicate rows and then set $A'=[\widetilde A,0]$, where $\widetilde A$ is the row-duplicated matrix.  A query $x=(x_0,x_{\rm extra})$ only uses $x_0$; a left query is aggregated over duplicated rows.  If $x\in\mathbb B_2^n$, then $x_0\in\mathbb B_2^{n_0}$, and
\[
  \max_i(A'x)_i=\max_i(Ax_0)_i,
  \qquad
  \|(A')^\top y\|_2=\|A^\top\bar y\|_2.
\]
Thus \eqref{eq:l21-gap} is preserved.
\end{proof}

\section{Probabilistic Preliminaries}\label{sec:inputs}
Only two external distributional facts are imported.  The elementary Haar projection estimate used later is proved below.

\begin{theorem}[Finite-dimensional Gaussian hard edge \cite{HuangTikhomirov}]\label{thm:hard-edge-input}
There are universal constants $c_{\rm he},C_{\rm he}>0$ such that, for every integer $k\ge 1$, every $s\in(0,1]$, and $Z_k\sim\Gin(k)$,
\begin{equation}\label{eq:hard-edge-input}
  c_{\rm he}s
  \le
  \Prob\!\left(\sigma_{\min}(Z_k)\le\frac{s}{\sqrt{k}}\right)
  \le
  C_{\rm he}s.
\end{equation}
\end{theorem}
Here $\Gin(k)$ denotes the $k\times k$ real matrix with independent $N(0,1)$ entries.  Equation \eqref{eq:hard-edge-input} is the specialization of the power estimate in \cite{HuangTikhomirov} to power exponent one; the symbol $k$ in the present manuscript denotes matrix dimension.

\begin{proposition}[Gaussian operator norm tail]\label{prop:gaussian-opnorm}
There are universal constants $C_{\rm op},c_{\rm op}>0$ such that, for every $N$ and $G\sim\Gin(N)$,
\[
  \Prob\bigl(\|G\|_{\op}>C_{\rm op}\sqrt N+t\bigr)
  \le 2e^{-c_{\rm op}t^2}
  \qquad(t\ge 0).
\]
In particular, for every fixed $\delta>0$, one can choose a universal $C(\delta)$ such that
$\Prob(\|G\|_{\op}>C(\delta)\sqrt N)\le\delta$ for all $N$.
\end{proposition}
This is the standard norm bound for a square Gaussian matrix; see, for example, \cite{Vershynin}.

\begin{lemma}[A random start sees any fixed direction]\label{lem:random-start}
There are universal constants $\nu,p_\nu>0$ such that, for every dimension $d\ge 2$, every fixed unit vector $v\in\R^d$, and a Haar-uniform unit vector $g\in\mathbb S^{d-1}$,
\[
  \Prob\!\left( |\langle g,v\rangle|\ge\frac{\nu}{\sqrt d}\right)
  \ge p_\nu.
\]
\end{lemma}
\begin{proof}
By rotational invariance take $v=e_1$ and write $g=z/\|z\|_2$ for $z\sim N(0,I_d)$.  Let $S=\sum_{j=2}^d z_j^2$, independent of $z_1$.  Markov's inequality gives $\Prob(S\le2(d-1))\ge1/2$, while $\Prob(1\le|z_1|\le2)$ is a fixed positive constant.  On their intersection,
\[
  \|z\|_2^2\le4+2(d-1)\le4d,
  \qquad
  |g_1|\ge\frac1{2\sqrt d}.
\]
Thus one may take $\nu=1/2$ and $p_\nu=\tfrac12\Prob(1\le|N(0,1)|\le2)>0$.
\end{proof}

\begin{lemma}[Haar one-dimensional upper tail]\label{lem:haar-tail}
Let $q$ be Haar-uniform on $\mathbb S^{d-1}$, where $d\ge2$, and let $a\in\R^d$ be fixed.  For every $t\in[0,1]$,
\begin{equation}\label{eq:haar-tail}
  \Prob\bigl(|\langle q,a\rangle|>t\|a\|_2\bigr)
  \le 2\exp\!\left(-\frac{d-1}{2}t^2\right).
\end{equation}
For $t>1$ the probability is zero.
\end{lemma}
\begin{proof}
For $a=0$ the strict tail event is empty.  By rotational invariance, assume $a=\|a\|_2e_1$ and write $q=g/\|g\|_2$ with $g\sim N(0,I_d)$.  Put $S=\sum_{j=2}^d g_j^2$, which is independent of $g_1$.  If $|q_1|>t$ with $0\le t<1$, then
\[
  |g_1|\ge \frac{t}{\sqrt{1-t^2}}\sqrt S.
\]
Using the Gaussian tail bound conditionally on $S$ and then the Laplace transform of a $\chi^2_{d-1}$ variable,
\[
\begin{aligned}
  \Prob(|q_1|>t)
  &\le 2\E\exp\!\left(-\frac{t^2}{2(1-t^2)}S\right)\\
  &=2(1-t^2)^{(d-1)/2}
  \le 2\exp\!\left(-\frac{d-1}{2}t^2\right).
\end{aligned}
\]
For $t\ge1$ the strict tail event is empty.
\end{proof}

\section{Adaptive Gaussian Spectral Hardness}\label{sec:gaussian-core}

\subsection{Two-sided Gaussian oracle and filtration}
Let $G\sim\Gin(N)$.  A two-sided query at $(u_j,v_j)\in\R^N\times\R^N$ returns
\[
  (Gv_j,G^\top u_j).
\]
The query is allowed to be any measurable function of the algorithm's internal random seed and the preceding transcript.  We include the seed in $\cF_0$ and assume it is independent of $G$.  For integers $a,b\ge0$, let $\gamma_{a,b}$ denote the product standard-Gaussian probability measure on $\R^{a\times b}$, and write $\gamma_a:=\gamma_{a,a}$.  Thus $Z\sim\Gin(a,b)$ means $\mathcal L(Z)=\gamma_{a,b}$.  Let $\cF_j$ be the sigma-field generated by the seed and the first $j$ query--response pairs.  Define
\[
  \cU_j:=\Span\{u_1,\ldots,u_j\},
  \qquad
  \cV_j:=\Span\{v_1,\ldots,v_j\},
\]
and write $r_j=\dim\cU_j$ and $s_j=\dim\cV_j$.

\paragraph{Measurable-frame convention.}
Every orthonormal frame and orthogonal completion below is selected by one fixed deterministic rule: scan the given ordered vectors, followed when necessary by $e_1,\ldots,e_N$, retain a vector exactly when it increases the current span, and apply Gram--Schmidt with a fixed sign convention.  On every fixed-rank stratum this construction is Borel; since there are only finitely many strata, all frames used below are measurable.  The same rule selects measurable frames for matrix row spaces, images, kernels, and their extensions to prescribed dimensions.

\begin{lemma}[Adaptive Gaussian posterior]\label{lem:posterior}
For each $j$, there are $\cF_j$-measurable frames
\[
  Q_{\cU_j}\in\R^{N\times r_j},\quad
  Q_{\cU_j^\perp}\in\R^{N\times(N-r_j)},\quad
  Q_{\cV_j}\in\R^{N\times s_j},\quad
  Q_{\cV_j^\perp}\in\R^{N\times(N-s_j)}
\]
with orthonormal columns spanning the indicated spaces.  If
\[
  P_{\cU_j}:=Q_{\cU_j}Q_{\cU_j}^\top,
  \qquad
  P_{\cV_j}:=Q_{\cV_j}Q_{\cV_j}^\top,
\]
then
\begin{equation}\label{eq:posterior-exact-decomposition}
  G
  =G_j^{\rm obs}
   +Q_{\cU_j^\perp}Z_jQ_{\cV_j^\perp}^\top,
  \qquad
  G_j^{\rm obs}:=P_{\cU_j}G+P_{\cU_j^\perp}GP_{\cV_j},
\end{equation}
where $G_j^{\rm obs}$ is $\cF_j$-measurable and the \emph{actual hidden block}
\[
  Z_j:=Q_{\cU_j^\perp}^\top GQ_{\cV_j^\perp}
\]
satisfies the regular conditional law
\begin{equation}\label{eq:posterior-exact-law}
  \mathcal L(Z_j\mid\cF_j)
  =\gamma_{N-r_j,N-s_j}
  \qquad\text{almost surely}.
\end{equation}
\end{lemma}

\begin{proof}
The transcript determines the linear maps $G|_{\cV_j}$ and $G^\top|_{\cU_j}$.  Consequently it determines $GP_{\cV_j}$ and $P_{\cU_j}G$, and hence $G_j^{\rm obs}$ is $\cF_j$-measurable.  The algebraic identity in \eqref{eq:posterior-exact-decomposition} follows from
\[
  G-P_{\cU_j}G-P_{\cU_j^\perp}GP_{\cV_j}
  =P_{\cU_j^\perp}GP_{\cV_j^\perp}.
\]
It remains to prove \eqref{eq:posterior-exact-law}.  This deferred-information argument is analogous to the adaptive posterior calculation for Wishart queries in \cite{BHSW20}; here we track separate left and right query spaces.

We argue by induction on $j$.  At $j=0$ the claim is immediate.  Assume it holds at $j-1$ and condition on $\cF_{j-1}$.  Put
\[
  \alpha:=Q_{\cU_{j-1}^\perp}^\top u_j,
  \qquad
  \beta:=Q_{\cV_{j-1}^\perp}^\top v_j.
\]
After subtracting the $\cF_{j-1}$-measurable contribution of $G_{j-1}^{\rm obs}$, the new response is equivalent to observing
\[
  Z_{j-1}\beta
  \qquad\text{and}\qquad
  Z_{j-1}^\top\alpha.
\]
Choose measurable orthogonal matrices $L,R$ that align each nonzero vector with the first coordinate:
\[
  L^\top\alpha=\|\alpha\|_2e_1,
  \qquad
  R^\top\beta=\|\beta\|_2e_1,
\]
and take the identity matrix on a zero vector.  Conditional on $\cF_{j-1}$,
\[
  \widetilde Z:=L^\top Z_{j-1}R
\]
is again a standard Gaussian matrix.  The new observations reveal precisely its first row when $\alpha\ne0$ and its first column when $\beta\ne0$ (with the common corner revealed only once).  All remaining entries form an independent standard Gaussian block.  Removing the revealed row and/or column and updating the two complement frames therefore gives \eqref{eq:posterior-exact-law} at time $j$.  If the prescribed frame convention selects different complement bases, their changes of coordinates are measurable with respect to the new transcript.  Conditional orthogonal invariance therefore preserves the standard Gaussian law in those bases.  All updates are measurable by the fixed frame convention.  This completes the induction.
\end{proof}

\subsection{Square-pivot augmentation and a fixed fresh Gaussian core}
The following argument appends analysis-only queries after the estimator has already produced its output.  Hence the output remains $\cF_t$-measurable.  The appended queries are used only to expose a convenient square pivot and to identify a residual Gaussian block.

\begin{lemma}[Adaptive Gaussian core]\label{lem:gaussian-core}
Let $t\in\{0,1,\ldots,\lfloor N/8\rfloor\}$, and suppose an adaptive algorithm makes at most $t$ two-sided queries to $G\sim\Gin(N)$.  Put
\[
  \ell:=N-4t.
\]
There exist an augmented sigma-field $\widetilde\cF_t\supseteq\cF_t$ and, on the same probability space, a random matrix $Z_\ell\in\R^{\ell\times\ell}$ such that
\begin{equation}\label{eq:conditional-core-law}
  \mathcal L(Z_\ell\mid\widetilde\cF_t)=\gamma_\ell
  \qquad\text{almost surely},
\end{equation}
and
\begin{equation}\label{eq:core-domination}
  \sigma_{\min}(G)\le\sigma_{\min}(Z_\ell)
  \qquad\text{almost surely}.
\end{equation}
Equivalently, for every bounded Borel function $\varphi:\R^{\ell\times\ell}\to\R$,
\[
  \E\!\left[\varphi(Z_\ell)\mid\widetilde\cF_t\right]
  =\int\varphi(z)\,\gamma_\ell(\dd z)
  \qquad\text{almost surely}.
\]
\end{lemma}

\begin{proof}
Pad the algorithm with zero queries if it stops before time $t$.  If $t=0$, take $\widetilde\cF_0=\cF_0$ and $Z_N=G$; the seed is independent of $G$, so the conclusion is immediate.  Hence assume $t\ge1$ and set
\[
  q:=2t,
  \qquad
  k:=N-q.
\]
Since $t\le N/8$, we have $q\le k$ and $\ell=k-q=N-4t\ge N/2$.

Let $\cU=\cU_t$ and $\cV=\cV_t$.  The space
\[
  \widehat\cU:=\cU+G\cV
\]
is $\cF_t$-measurable, because $G\cV$ is spanned by already observed column responses.  On the probability-one event that $G$ is invertible,
\[
  \dim\widehat\cU
  \le \dim\cU+\dim\cV
  \le 2t=q.
\]
Use the measurable-frame convention to extend $\widehat\cU$ to a $q$-dimensional $\cF_t$-measurable subspace $\overline\cU$, with orthonormal frame $Q_{\overline\cU}$.  After the estimator has output, append the $q$ row-only queries
\[
  (Q_{\overline\cU}e_i,0),
  \qquad i=1,\ldots,q.
\]
They reveal the full row map
\[
  M:=Q_{\overline\cU}^\top G\in\R^{q\times N}.
\]
Because $G$ is invertible, $M$ has row rank $q$.  Moreover, $M|_{\cV}$ is injective: if $v\in\cV$ and $Mv=0$, then $Gv\in G\cV\subseteq\overline\cU$, while its coordinates in $\overline\cU$ vanish, so $Gv=0$ and hence $v=0$.

Starting from a measurable orthonormal basis of $\cV$, scan $e_1,\ldots,e_N$ and retain a vector exactly when its $M$-image increases the span of the previously retained $M$-images.  Since $M$ has rank $q$ and $M|_{\cV}$ is injective, the procedure stops with $q$ vectors.  After Gram--Schmidt they span a $q$-dimensional measurable subspace $\overline\cV\supseteq\cV$ on which $M$ is bijective.  Let $Q_{\overline\cV}$ be its selected orthonormal frame, and append the $q$ column-only queries
\[
  (0,Q_{\overline\cV}e_i),
  \qquad i=1,\ldots,q.
\]
Let $\cG_t$ be the sigma-field generated by $\cF_t$ and these $2q$ post-output query--response pairs.

Choose measurable orthogonal complements $Q_{\overline\cU^\perp}$ and $Q_{\overline\cV^\perp}$.  In the corresponding block coordinates, write the actual matrix as
\begin{equation}\label{eq:block-G}
  \begin{pmatrix}
    Q_{\overline\cU}^\top\\
    Q_{\overline\cU^\perp}^\top
  \end{pmatrix}
  G
  \begin{pmatrix}
    Q_{\overline\cV}&Q_{\overline\cV^\perp}
  \end{pmatrix}
  =
  \begin{pmatrix}
    A&B\\
    C&Z
  \end{pmatrix}.
\end{equation}
The appended row queries reveal $A,B$, and the appended column queries reveal $A,C$, so these three blocks are $\cG_t$-measurable.  Also
\[
  A=M Q_{\overline\cV}
\]
is invertible.  Applying \cref{lem:posterior} to the complete augmented transcript gives
\begin{equation}\label{eq:hidden-square-law}
  \mathcal L(Z\mid\cG_t)=\gamma_k
  \qquad\text{almost surely}.
\end{equation}

Define the Schur complement
\[
  S:=Z-CA^{-1}B.
\]
For every unit vector $z\in\R^k$, let $x_z=(-A^{-1}Bz,z)$ in the block domain of \eqref{eq:block-G}.  Then
\[
  \begin{pmatrix}A&B\\ C&Z\end{pmatrix}x_z
  =\binom{0}{Sz},
  \qquad
  \|x_z\|_2\ge1.
\]
Since the coordinate changes in \eqref{eq:block-G} are orthogonal,
\begin{equation}\label{eq:schur-smin}
  \sigma_{\min}(G)
  \le \inf_{\|z\|_2=1}\frac{\|Sz\|_2}{\|x_z\|_2}
  \le \sigma_{\min}(S).
\end{equation}

Put $F:=CA^{-1}B$.  Since $C$ has $q$ columns, $\rank(F)\le q$.  Select measurably a $q$-dimensional subspace $\mathcal L\subseteq\R^k$ containing $\operatorname{im}(F)$, and let $O\in O(k)$ be the selected orthogonal change of basis that sends $\mathcal L$ to the first $q$ coordinate directions.  Then
\[
  OF=\binom{F_1}{0},
  \qquad
  F_1\in\R^{q\times k}.
\]
Conditional on $\cG_t$, $O$ is fixed; hence \eqref{eq:hidden-square-law} and left-orthogonal invariance give
\[
  OZ=\binom{Z_1}{W},
  \qquad
  Z_1\in\R^{q\times k},
  \quad
  W\in\R^{\ell\times k},
\]
where, conditional on $\cG_t$, $Z_1$ and $W$ are independent standard Gaussian blocks.  Set
\[
  Y:=Z_1-F_1.
\]
By the measurable-frame convention, choose a $\sigma(Y)$-measurable matrix $V_0(Y)\in\R^{k\times\ell}$ with orthonormal columns contained in $\ker(Y)$.  This is always possible because
\[
  \dim\ker(Y)\ge k-q=\ell.
\]
Define
\[
  \widetilde\cF_t:=\cG_t\vee\sigma(Y),
  \qquad
  Z_\ell:=WV_0(Y).
\]
Conditional on $\widetilde\cF_t$, the matrix $V_0(Y)$ is fixed and $W$ remains an $\ell\times k$ standard Gaussian matrix.  Each row of $WV_0(Y)$ is therefore $N(0,I_\ell)$, and the rows are independent.  This proves the regular conditional identity \eqref{eq:conditional-core-law}.

Finally, $YV_0=0$ gives
\[
  OSV_0
  =O(Z-F)V_0
  =\binom{0}{WV_0}
  =\binom{0}{Z_\ell}.
\]
Since $O$ is orthogonal and $V_0$ is an isometric embedding,
\begin{align*}
  \sigma_{\min}(S)
  &=\inf_{\|x\|_2=1}\|OSx\|_2\\
  &\le\inf_{\|u\|_2=1}\|OSV_0u\|_2\\
  &=\inf_{\|u\|_2=1}\|Z_\ell u\|_2
   =\sigma_{\min}(Z_\ell).
\end{align*}
Combining this with \eqref{eq:schur-smin} proves \eqref{eq:core-domination}.  All constructions may be defined arbitrarily on the null event that $G$ is singular.
\end{proof}

\subsection{Uniform hard-edge separation}
\Cref{lem:gaussian-core} supplies a fresh square matrix of dimension at least $N/2$, regardless of the adaptive transcript.  To turn that structural fact into estimation hardness, we need a lower-tail event uniform over every such core dimension, together with an unconditional high event for the original matrix.  The following separation makes those two requirements explicit.

\begin{lemma}[Separated low and high events]\label{lem:hard-edge-separation}
There are universal constants $0<a<b$, $p_->0$, and $p_+>0$, with $b>8a$, such that for every $N$:
\begin{align}
  \inf_{\ell\in[N/2,N]}
  \Prob\!\left(\sigma_{\min}(Z_\ell)\le\frac{a}{\sqrt N}\right)
  &\ge p_-,\label{eq:low-uniform}\\
  \Prob\!\left(\sigma_{\min}(G_N)\ge\frac{b}{\sqrt N}\right)
  &\ge p_+,\label{eq:high-uniform}
\end{align}
where $Z_\ell\sim\Gin(\ell)$ and $G_N\sim\Gin(N)$.
\end{lemma}

\begin{proof}
Let $c_{\rm he},C_{\rm he}$ be from \cref{thm:hard-edge-input}.  Set
\[
  b:=\min\left\{1,\frac{1}{4C_{\rm he}}\right\},
  \qquad
  a:=\frac{b}{16}.
\]
Then \eqref{eq:hard-edge-input} gives
\[
  \Prob\!\left(\sigma_{\min}(G_N)<\frac{b}{\sqrt N}\right)
  \le C_{\rm he}b\le\frac14,
\]
so $p_+=3/4$ works.  If $\ell\in[N/2,N]$, write
\[
  \frac{a}{\sqrt N}=\frac{s_\ell}{\sqrt\ell},
  \qquad
  s_\ell=a\sqrt{\frac{\ell}{N}}\in\left[\frac{a}{\sqrt2},a\right].
\]
The lower bound in \eqref{eq:hard-edge-input} yields
\[
  \Prob\!\left(\sigma_{\min}(Z_\ell)\le\frac{a}{\sqrt N}\right)
  \ge c_{\rm he}s_\ell
  \ge\frac{c_{\rm he}a}{\sqrt2}.
\]
Thus $p_-=c_{\rm he}a/\sqrt2$ works, and $b=16a>8a$.
\end{proof}

\begin{theorem}[Adaptive Gaussian smallest-singular-value lower bound]\label{thm:spectral-estimation}
There is a universal $\gamma>0$ such that, for every $N$ and every randomized estimator that uses at most $t\le N/8$ adaptive two-sided queries to $G\sim\Gin(N)$ and returns $\widehat s\in\R$, one has
\begin{equation}\label{eq:spectral-lb}
  \Prob\!\left(
    \widehat s\notin
    \left[\frac12\sigma_{\min}(G),\,2\sigma_{\min}(G)\right]
  \right)
  \ge\gamma.
\end{equation}
\end{theorem}

\begin{proof}
Include the estimator's random seed in $\cF_0$.  Apply \cref{lem:gaussian-core}; since $t\le N/8$, the resulting $\ell$ lies in $[N/2,N]$.  By \cref{lem:hard-edge-separation}, for almost every augmented transcript,
\begin{equation}\label{eq:conditional-low}
  \Prob\!\left(
    \sigma_{\min}(G)\le\frac{a}{\sqrt N}
    \ \middle|\ \widetilde\cF_t
  \right)
  \ge p_-.
\end{equation}
The implication follows because
$\sigma_{\min}(G)\le\sigma_{\min}(Z_\ell)$ under the conditional coupling.

Let
\[
  \tau:=\frac{b}{4\sqrt N},
  \qquad
  E:=\{\widehat s\ge\tau\}.
\]
Because the estimator's complete random seed is included in $\cF_0$, its output and hence $E$ are $\cF_t$-measurable, and therefore also $\widetilde\cF_t$-measurable.
On $E\cap\{\sigma_{\min}(G)\le a/\sqrt N\}$, the estimate exceeds $2\sigma_{\min}(G)$ because $b=16a$.  Hence, by \eqref{eq:conditional-low},
\[
  \Prob(\mathrm{fail})\ge p_-\Prob(E).
\]
On $E^c\cap\{\sigma_{\min}(G)\ge b/\sqrt N\}$, the estimate is below $\sigma_{\min}(G)/2$, so
\[
  \Prob(\mathrm{fail})
  \ge p_+-\Prob(E).
\]
Minimizing the maximum of these two lower bounds over $q=\Prob(E)\in[0,1]$ gives
\[
  \Prob(\mathrm{fail})
  \ge
  \frac{p_-p_+}{1+p_-}
  =:\gamma>0.
\]
\end{proof}

\section{From Residual Solving to Inverse-Norm Estimation}\label{sec:solve}

\Cref{thm:spectral-estimation} rules out highly reliable estimation after a sublinear number of Gaussian queries.  We now show why a uniformly fast residual solver would violate that result.  Residual verification first amplifies reliability, refinement converts a constant contraction into an accurate inverse application, and approximate inverse power estimates the inverse norm.

\subsection{Symmetric indefinite systems}
For $d\ge 1$ and $\kappa\ge 1$, define
\[
  \cH_{d,\kappa}
  :=\left\{
    H=H^\top\in\R^{d\times d}:
    \sigma_{\min}(H)\ge 1,
    \ \|H\|_{\op}\le\kappa
  \right\}.
\]
Equivalently, the spectrum lies in $[-\kappa,-1]\cup[1,\kappa]$.

\begin{definition}[Constant-residual solver]
A randomized $(K,3/5)$ constant-residual solver, given $H\in\cH_{d,\kappa}$ and $r\in\R^d$, uses at most $K$ queries to $z\mapsto Hz$ and returns a vector $x$ such that
\[
  \Prob\!\left(\|Hx-r\|_2\le\frac12\|r\|_2\right)\ge\frac35.
\]
\end{definition}

\subsection{Verified amplification and iterative refinement}

\begin{lemma}[Verified amplification]\label{lem:amplify}
Suppose a base solver has success probability at least $3/5$ and uses at most $K$ matvecs.  By running it independently $R$ times, using one extra matvec per candidate to compute its exact residual, and selecting the candidate of minimum residual, the failure probability is at most $(2/5)^R$ and the total query count is at most $R(K+1)$.
\end{lemma}
\begin{proof}
If at least one run succeeds, the candidate with minimum verified residual succeeds.  All runs fail with probability at most $(2/5)^R$.
\end{proof}

\begin{lemma}[Accurate inverse application by refinement]\label{lem:refinement}
Let $H\in\cH_{d,\kappa}$ and $M=H^{-1}$, $L=\|M\|_{\op}$.  Suppose every amplified correction succeeds.  Then, after
\[
  J=\left\lceil\log_2(1/\eta)\right\rceil
\]
corrections, one can return $\widetilde x$ satisfying
\begin{equation}\label{eq:inverse-application}
  \|\widetilde x-Mr\|_2
  \le \eta L\|r\|_2.
\end{equation}
\end{lemma}
\begin{proof}
Initialize $x_0=0$ and $r_0=r$.  At stage $j$, obtain $\Delta_j$ with
\[
  \|H\Delta_j-r_j\|_2\le\frac12\|r_j\|_2,
\]
and set
\[
  x_{j+1}=x_j+\Delta_j,
  \qquad
  r_{j+1}=r_j-H\Delta_j.
\]
The product $H\Delta_j$ is already available from candidate verification and is reused in the update.  If $r_j=0$, return the current $x_j$ immediately.  Then $\|r_J\|_2\le 2^{-J}\|r\|_2\le\eta\|r\|_2$, while
\[
  \widetilde x-Mr=x_J-Mr=-M r_J.
\]
Taking norms yields \eqref{eq:inverse-application}.
\end{proof}

\subsection{Approximate inverse power}

\begin{lemma}[Inverse-norm estimation from accurate inverse applications]\label{lem:inverse-power}
Let $d\ge2$, let $H=H^\top\in\R^{d\times d}$ be invertible, and set $M=H^{-1}$ and $L=\|M\|_{\op}$.  Fix universal constants $\nu,p_\nu$ from \cref{lem:random-start}.  There are universal constants $C_1,C_2$ such that the following holds.

Take
\[
  q=\left\lceil C_1\log(d+2)\right\rceil,
  \qquad
  \eta=\frac{\nu}{C_2q\sqrt d}.
\]
Use a constant number $S=O(1)$ of independent Haar starts $g^{(s)}$, set $z_0^{(s)}:=g^{(s)}$, and for each start recursively compute $q$ approximate inverse applications satisfying \eqref{eq:inverse-application} with input $r=z_j^{(s)}$.  If every application succeeds and at least one start obeys
$|\langle g^{(s)},v_*\rangle|\ge\nu/\sqrt d$ for a top absolute-eigenvalue eigenvector $v_*$ of $M$, then
\[
  \widehat L:=\max_{s\le S}\|z_q^{(s)}\|_2^{1/q}
\]
satisfies
\begin{equation}\label{eq:L-factor-two}
  \frac12L\le\widehat L\le2L.
\end{equation}
\end{lemma}

\begin{proof}
For one start, set $e_j=\|z_j-M^jg\|_2$ and $a_j=e_j/L^j$.  The application error obeys
\[
  e_{j+1}
  \le \eta L\|z_j\|_2+Le_j
  \le \eta L^{j+1}+(1+\eta)Le_j,
\]
so
\[
  a_{j+1}\le\eta+(1+\eta)a_j,
  \qquad
  a_0=0.
\]
Thus
\[
  a_q\le(1+\eta)^q-1\le 2q\eta
\]
when $q\eta\le1$.  Taking $C_2$ large enough gives
\[
  e_q\le\frac{\nu}{4\sqrt d}L^q.
\]
For a good start,
\[
  \|M^qg\|_2\ge\frac{\nu}{\sqrt d}L^q,
\]
and therefore
\[
  \|z_q\|_2\ge\frac{3\nu}{4\sqrt d}L^q.
\]
For every start,
\[
  \|z_q\|_2\le L^q+e_q\le2L^q.
\]
Choose $C_1$ so that
\[
  \left(\frac{3\nu}{4\sqrt d}\right)^{1/q}\ge\frac12.
\]
Taking $q$-th roots and then the maximum over starts proves \eqref{eq:L-factor-two}.
\end{proof}

\subsection{Linear-dimensional randomized solve lower bound}

\begin{theorem}[Ordinary-matvec lower bound for indefinite solving]\label{thm:solve-lb}
There are universal constants $c,C>0$ such that, for every sufficiently large $N$, with
\[
  d=2N,
  \qquad
  \kappa=CN,
\]
any randomized constant-residual solver for $\cH_{d,\kappa}$ must use
\begin{equation}\label{eq:solve-lb}
  K\ge
  c\frac{N}{\log^2(N+2)\,\log\log(N+e^e)}
\end{equation}
matvec queries in the worst case.
\end{theorem}

\begin{proof}
Assume for contradiction that a solver with smaller $K$ exists.  By the hard-cap convention following the query-complexity definition, we run a version that is stopped after $K$ matvecs and returns zero if it has not terminated; this changes nothing on promised inputs.  We use it to contradict \cref{thm:spectral-estimation}.

Given two-sided oracle access to $G\sim\Gin(N)$, form the symmetric dilation
\[
  \mathcal D(G):=
  \begin{pmatrix}
    0&G\\
    G^\top&0
  \end{pmatrix}.
\]
A query to $\mathcal D(G)$ at $(u,v)\in\R^N\times\R^N$ returns
\[
  \mathcal D(G)\binom{u}{v}
  =\binom{Gv}{G^\top u},
\]
which is exactly one two-sided query to $G$.

Let $\gamma$ be from \cref{thm:spectral-estimation}.  By the upper bound in \cref{thm:hard-edge-input}, choose a sufficiently small universal $a_*>0$ so that
\[
  \Prob\!\left(\sigma_{\min}(G)<\frac{a_*}{\sqrt N}\right)
  \le\frac{\gamma}{64}.
\]
By \cref{prop:gaussian-opnorm}, choose a universal $C_*>0$ such that
\[
  \Prob(\|G\|_{\op}>C_*\sqrt N)
  \le\frac{\gamma}{64}.
\]
Define
\[
  H:=\frac{\sqrt N}{a_*}\mathcal D(G),
  \qquad
  \kappa:=\frac{C_*}{a_*}N.
\]
On the event
\[
  \mathcal E_N:=\left\{
    \sigma_{\min}(G)\ge\frac{a_*}{\sqrt N},
    \ \|G\|_{\op}\le C_*\sqrt N
  \right\},
\]
we have $H\in\cH_{2N,\kappa}$, and $\Prob(\mathcal E_N^c)\le\gamma/32$.

Let $M=H^{-1}$ and $L=\|M\|_{\op}$.  Use \cref{lem:inverse-power}, implementing each accurate inverse application by \cref{lem:refinement} and each correction by the amplified procedure of \cref{lem:amplify}.  Let $q=O(\log N)$ be the inverse-power length and let $\eta=\Theta(1/(\sqrt N\log N))$ be its required inverse-application accuracy.  Each accurate application uses
\[
  J=O(\log(1/\eta))=O(\log N)
\]
correction steps.  Across the constant number of starts, the number of amplified calls is
\[
  M_{\rm calls}=O(\log^2 N).
\]
Choose
\[
  R=O(\log(M_{\rm calls}/\gamma))=O(\log\log N)
\]
independent base-solver repetitions per correction.  Condition on $G$ with $\mathcal E_N$ holding and on the complete history before any particular correction.  Its right-hand side is then fixed, and fresh independent solver seeds make the amplified correction fail with probability at most $(2/5)^R$.  Applying the tower property and a union bound over the adaptively generated correction calls, all corrections succeed with probability at least $1-\gamma/32$ conditioned on $\mathcal E_N$.  By \cref{lem:random-start}, a sufficiently large universal constant number of starts gives a good start with probability at least $1-\gamma/32$.

Consequently, with total failure probability at most
\[
  \Prob(\mathcal E_N^c)+\frac{\gamma}{32}+\frac{\gamma}{32}
  \le\frac{3\gamma}{32}<\gamma,
\]
we obtain $\widehat L\in[L/2,2L]$.  Since
\[
  L=\frac{a_*}{\sqrt N\,\sigma_{\min}(G)},
\]
the output
\[
  \widehat\sigma:=
  \begin{cases}
    a_*/(\sqrt N\,\widehat L),&\widehat L>0,\\
    0,&\widehat L=0
  \end{cases}
\]
is a factor-two estimate on success and has a defined value on every transcript.

The total number of two-sided Gaussian queries is at most
\[
  C'(K+1)\log^2(N+2)\,\log\log(N+e^e).
\]
If the constant in the assumed upper bound on $K$ is sufficiently small, this is less than $N/8$, contradicting \cref{thm:spectral-estimation}.  Hence \eqref{eq:solve-lb} holds.
\end{proof}

\section{From Linear Systems to Matrix Games}\label{sec:embedding}

Throughout this section fix $d\ge2$, $H\in\cH_{d,\kappa}$, and a unit right-hand side $b\in\R^d$.  The two constructions address different geometric obstacles: the ball permits a direct homogeneous representation of the solution, whereas simplex strategies require a signed dictionary.  In each case we verify admissibility, recover a constant-residual solution from a small full gap, and simulate one game query with two source matvecs.  The construction uses only matvec access to $H$.

\subsection[A direct l2--l1 embedding]{A direct $\ell_2$--$\ell_1$ embedding}\label{subsec:l21-embedding}
Set
\begin{equation}\label{eq:l21-alpha-beta}
  \beta:=\frac1{4\kappa},
  \qquad
  \alpha:=\frac{\beta}{8\sqrt d}.
\end{equation}
For $i\in[d]$ and $\sigma\in\{-1,+1\}$ define a row vector in $\R^{d+1}$ by
\begin{equation}\label{eq:l21-C-row}
  C_{(i,\sigma),:}
  :=\alpha e_{d+1}^\top
    +\sigma\beta\,(e_i^\top H,-b_i),
\end{equation}
and set
\begin{equation}\label{eq:l21-A}
  A^{(2,1)}:=-C\in\R^{2d\times(d+1)}.
\end{equation}
Write a minimization vector as $w=(z,s)\in\R^d\times\R$.

\begin{lemma}[Row-norm normalization]\label{lem:l21-row-norm}
Every row of $A^{(2,1)}$ has Euclidean norm strictly less than one.  Hence $A^{(2,1)}$ is an admissible $\ell_2$--$\ell_1$ payoff matrix.
\end{lemma}
\begin{proof}
Using $\|He_i\|_2\le\kappa$, $|b_i|\le1$, and $\kappa\ge1$,
\[
  \|C_{(i,\sigma),:}\|_2
  \le\alpha+\beta\sqrt{\kappa^2+1}
  \le\frac1{32}+\frac{\sqrt2}{4}<1.
\]
Changing the sign of the matrix does not change row norms.
\end{proof}

Define
\begin{equation}\label{eq:l21-psi}
  \psi(z,s)
  :=\min_{i,\sigma}C_{(i,\sigma),:}(z,s)
  =\alpha s-\beta\|Hz-sb\|_\infty.
\end{equation}
Since $A^{(2,1)}=-C$, the primal objective of the matrix game is
\begin{equation}\label{eq:l21-primal}
  F(z,s)
  :=\max_{y\in\Delta_{2d}}y^\top A^{(2,1)}(z,s)
  =-\psi(z,s).
\end{equation}

\begin{lemma}[A positive-margin witness]\label{lem:l21-witness}
Let $z_*=H^{-1}b$ and
\[
  s_*:=\frac1{\sqrt{1+\|z_*\|_2^2}},
  \qquad
  \bar z_*:=s_*z_*.
\]
Then $(\bar z_*,s_*)\in\mathbb B_2^{d+1}$ and
\[
  \psi(\bar z_*,s_*)=\alpha s_*\ge\frac\alpha{\sqrt2}.
\]
Consequently, if $v_*$ denotes the value of the game with payoff $A^{(2,1)}$, then $v_*\le-\alpha/\sqrt2$.
\end{lemma}
\begin{proof}
Because $\sigma_{\min}(H)\ge1$ and $\|b\|_2=1$, we have $\|z_*\|_2\le1$.  By construction, $\|(\bar z_*,s_*)\|_2=1$ and
\[
  H\bar z_*-s_*b=0.
\]
Equation \eqref{eq:l21-psi} gives the claimed value, and minimizing \eqref{eq:l21-primal} yields the value bound.
\end{proof}

\begin{lemma}[Small full gap yields a constant-residual solution]\label{lem:l21-small-gap}
Let $(w,y)=((z,s),y)\in\mathbb B_2^{d+1}\times\Delta_{2d}$.  If
\begin{equation}\label{eq:l21-gap-threshold}
  \Gap_{A^{(2,1)}}^{2,1}(w,y)\le\frac\alpha{16},
\end{equation}
then
\[
  s\ge\frac1{\sqrt2}-\frac1{16}>0
\]
and the vector $\widehat z:=z/s$ satisfies
\begin{equation}\label{eq:l21-residual}
  \|H\widehat z-b\|_2<\frac12.
\end{equation}
\end{lemma}
\begin{proof}
For every $y$, the dual value $\min_{w'\in\mathbb B_2^{d+1}}y^\top A^{(2,1)}w'$ is at most the game value $v_*$.  Hence the full gap dominates primal suboptimality:
\[
  \Gap_{A^{(2,1)}}^{2,1}(w,y)
  \ge F(w)-v_*.
\]
By \cref{lem:l21-witness} and \eqref{eq:l21-primal}, \eqref{eq:l21-gap-threshold} implies
\begin{equation}\label{eq:l21-psi-lower}
  \psi(z,s)\ge\frac\alpha{\sqrt2}-\frac\alpha{16}.
\end{equation}
Because the residual term in \eqref{eq:l21-psi} is nonnegative, \eqref{eq:l21-psi-lower} gives
\[
  s\ge\frac1{\sqrt2}-\frac1{16}=:s_0>0.
\]
Since $|s|\le1$ for $w\in\mathbb B_2^{d+1}$, equations \eqref{eq:l21-psi} and \eqref{eq:l21-psi-lower} also give
\[
  \beta\|Hz-sb\|_\infty
  \le\alpha\left(1-\frac1{\sqrt2}+\frac1{16}\right).
\]
Therefore, using \eqref{eq:l21-alpha-beta},
\[
  \|Hz-sb\|_2
  \le\sqrt d\,\|Hz-sb\|_\infty
  \le\frac18\left(1-\frac1{\sqrt2}+\frac1{16}\right)
  <0.045.
\]
Dividing by $s\ge s_0>0.64$ gives \eqref{eq:l21-residual}.
\end{proof}

\begin{lemma}[Exact $\ell_2$--$\ell_1$ oracle simulation]\label{lem:l21-oracle-simulation}
For arbitrary $w=(z,s)\in\R^{d+1}$,
\begin{equation}\label{eq:l21-sim-forward}
  Cw
  =\alpha s\one_{2d}
   +\beta
   \begin{pmatrix}
     Hz-sb\\
     -(Hz-sb)
   \end{pmatrix}.
\end{equation}
For arbitrary $y=(y^+,y^-)\in\R^{2d}$, let $q:=y^+-y^-$.  Then
\begin{equation}\label{eq:l21-sim-transpose}
  C^\top y
  =\begin{pmatrix}
      \beta Hq\\
      \alpha\one^\top y-\beta b^\top q
    \end{pmatrix}.
\end{equation}
Thus one two-sided query to $A^{(2,1)}=-C$ is simulated by two ordinary queries to $H$.
\end{lemma}
\begin{proof}
Both identities follow by expanding \eqref{eq:l21-C-row} and using $H=H^\top$.
\end{proof}

\begin{proof}[Proof of the $\ell_2$--$\ell_1$ part of \cref{thm:main-sequences}]
Take $d=2N$ and $\kappa=CN$ as in \cref{thm:solve-lb}.  Suppose a randomized $\ell_2$--$\ell_1$ algorithm uses $T$ two-sided queries and, for every admissible $2d\times(d+1)$ payoff matrix, returns full gap at most $\alpha/16$ with probability at least $2/3$.

For an arbitrary source pair $H\in\cH_{d,\kappa}$ and right-hand side $b$, return zero immediately when $b=0$.  Otherwise normalize $b_0=b/\|b\|_2$, construct $A^{(2,1)}(H,b_0)$, and simulate every game query by \cref{lem:l21-oracle-simulation}.  If the game output is $((z,s),y)$, return
\[
  \widehat z_{\rm src}:=
  \begin{cases}
    \|b\|_2z/s,&s>0,\\
    0,&s\le0.
  \end{cases}
\]
On the game-success event, \cref{lem:l21-small-gap} ensures $s>0$ and
\[
  \|H\widehat z_{\rm src}-b\|_2
  <\frac12\|b\|_2.
\]
Hence we obtain a source solver with success probability at least $2/3>3/5$ and at most $2T$ ordinary matvecs.  By \cref{thm:solve-lb},
\[
  2T\ge c\frac{N}{\log^2(N+2)\,\log\log(N+e^e)}.
\]
Finally, with $\kappa=CN$ and $d=2N$,
\[
  \frac\alpha{16}
  =\frac{1}{512C\sqrt2}\,N^{-3/2}.
\]
Thus part (a) holds with $c_0=1/(512C\sqrt2)$ and with $c_2$ chosen no larger than half the source constant in \cref{thm:solve-lb}.
\end{proof}

\subsection[A bounded l1--l1 embedding]{A bounded $\ell_1$--$\ell_1$ embedding}\label{subsec:l11-embedding}
The homogeneous ball construction does not directly produce simplex strategies.  Instead, one signed dictionary represents every unit-ball solution as a convex combination, while a second dictionary tests every residual direction.  Randomizing only the testing basis spreads the matrix entries enough to satisfy the entrywise bound, with the logarithmic cost recorded below.
Let $Q\in O(d)$ be Haar-uniform and define
\[
  R:=\sqrt d\,[Q,-Q]\in\R^{d\times2d},
  \qquad
  V:=\sqrt d\,[I,-I]\in\R^{d\times2d}.
\]
Set
\begin{equation}\label{eq:l11-A-embedding}
  A^{(1,1)}=A^{(1,1)}(H,b,Q)
  :=\frac1\Lambda
  \left(R^\top HV-(R^\top b)\one^\top\right)
  \in\R^{2d\times2d},
\end{equation}
where
\begin{equation}\label{eq:l11-Lambda}
  \Lambda:=C_{\rm emb}\kappa\sqrt{d\log(d+2)}
\end{equation}
and $C_{\rm emb}$ is a sufficiently large universal constant.

\begin{lemma}[Entrywise boundedness with constant probability]\label{lem:l11-bounded-entries}
For every fixed $H\in\cH_{d,\kappa}$ and unit $b$, if $C_{\rm emb}$ is large enough then
\[
  \Prob_Q\bigl(\|A^{(1,1)}(H,b,Q)\|_{\max}\le1\bigr)\ge0.99.
\]
\end{lemma}
\begin{proof}
Let $q_i$ be a column of $Q$.  Each $q_i$ is marginally Haar-uniform on $\mathbb S^{d-1}$; independence of the columns is not needed for a union bound.  Up to signs, the bilinear entries of $R^\top HV$ equal $d\,q_i^\top He_j$.  Fix a sufficiently large universal $C$ and put
\[
  t_d:=C\sqrt{\frac{\log(d+2)}{d}}.
\]
When $t_d\le1$, \cref{lem:haar-tail} and $(d-1)/d\ge1/2$ give, for each pair $(i,j)$,
\[
  \Prob\!\left(|q_i^\top He_j|>t_d\kappa\right)
  \le 2(d+2)^{-C^2/4}.
\]
When $t_d>1$, the same event is impossible because $|q_i^\top He_j|\le\|He_j\|_2\le\kappa$.  Choosing $C$ so that $2d^2(d+2)^{-C^2/4}\le0.005$ for every $d\ge2$ and taking a union bound yields
\[
  \max_{i,j}d|q_i^\top He_j|
  \le C\kappa\sqrt{d\log(d+2)}
\]
with probability at least $0.995$.  The identical argument with $He_j$ replaced by the unit vector $b$ and a union bound over $i$ gives
\[
  \max_i\sqrt d|q_i^\top b|
  \le C\sqrt{\log(d+2)}
\]
with probability at least $0.995$.  A final union bound gives probability at least $0.99$ for both displays.  Since $\kappa\ge1$ and $d\ge2$, the sum of the two entrywise contributions is at most $C'\kappa\sqrt{d\log(d+2)}$ for a universal $C'$, and is absorbed by \eqref{eq:l11-Lambda} once $C_{\rm emb}\ge C'$.
\end{proof}

\begin{lemma}[Simplex witness]\label{lem:l11-simplex-witness}
There exists $x_*\in\Delta_{2d}$ such that
\[
  Vx_*=H^{-1}b
  \qquad\text{and}\qquad
  A^{(1,1)}x_*=0.
\]
\end{lemma}
\begin{proof}
Since $\|H^{-1}b\|_2\le1$, we have $\|H^{-1}b\|_1\le\sqrt d$.  The convex hull of the columns $\{\pm\sqrt d\,e_j\}_{j=1}^d$ of $V$ is the $\ell_1$ ball of radius $\sqrt d$.  Hence $H^{-1}b=Vx_*$ for some $x_*\in\Delta_{2d}$.  Since $\one^\top x_*=1$, equation \eqref{eq:l11-A-embedding} gives $A^{(1,1)}x_*=0$.
\end{proof}

\begin{lemma}[The full $\ell_1$--$\ell_1$ gap controls the residual]\label{lem:l11-residual-gap}
For every $(x,y)\in\Delta_{2d}\times\Delta_{2d}$,
\begin{equation}\label{eq:l11-residual-gap}
  \Gap_{A^{(1,1)}}^{1,1}(x,y)
  \ge\frac1\Lambda\|HVx-b\|_2.
\end{equation}
\end{lemma}
\begin{proof}
Because $R$ contains both signs of the orthonormal basis $Q$,
\[
  \max_i(R^\top z)_i
  =\sqrt d\,\|Q^\top z\|_\infty
  \ge\|z\|_2.
\]
For $x\in\Delta_{2d}$,
\[
  A^{(1,1)}x=\frac1\Lambda R^\top(HVx-b),
\]
so $\max_i(A^{(1,1)}x)_i\ge\Lambda^{-1}\|HVx-b\|_2$.  By \cref{lem:l11-simplex-witness},
\[
  \min_j((A^{(1,1)})^\top y)_j
  \le x_*^\top(A^{(1,1)})^\top y=0.
\]
Substitution into \eqref{eq:l11-gap} proves the claim.
\end{proof}

\begin{lemma}[Exact $\ell_1$--$\ell_1$ oracle simulation]\label{lem:l11-oracle-simulation}
For arbitrary real query vectors $x,y\in\R^{2d}$,
\begin{align}
  A^{(1,1)}x
  &=\frac1\Lambda R^\top\left(H(Vx)-b\,\one^\top x\right),\label{eq:l11-sim-Ax}\\
  (A^{(1,1)})^\top y
  &=\frac1\Lambda\left(V^\top H(Ry)-\one\,b^\top Ry\right).\label{eq:l11-sim-ATy}
\end{align}
Thus one two-sided query to $A^{(1,1)}$ is simulated by two ordinary queries to $H$.
\end{lemma}
\begin{proof}
Expand \eqref{eq:l11-A-embedding} and use $H=H^\top$.
\end{proof}

\begin{proof}[Proof of the $\ell_1$--$\ell_1$ part of \cref{thm:main-sequences}]
Take $d=2N$ and $\kappa=CN$ as in \cref{thm:solve-lb}.  Suppose a randomized $\ell_1$--$\ell_1$ algorithm uses $T$ two-sided queries and, for every bounded $2d\times2d$ payoff matrix, returns full gap at most $1/(2\Lambda)$ with probability at least $2/3$.

For an arbitrary source pair $H\in\cH_{d,\kappa}$ and right-hand side $b$, return zero immediately when $b=0$.  Otherwise normalize $b_0=b/\|b\|_2$, sample $Q$ Haar-uniformly independently of the game algorithm's random seed, and run the game algorithm on the implicit matrix $A^{(1,1)}(H,b_0,Q)$, simulating each game query using \cref{lem:l11-oracle-simulation}.  Output
\[
  \widehat z:=\|b\|_2V\widehat x.
\]
For every fixed $H,b_0$, \cref{lem:l11-bounded-entries} gives a good-$Q$ event of probability at least $0.99$.  For each fixed $Q$ in this event, the game matrix is admissible and the game algorithm succeeds with probability at least $2/3$.  On success, \cref{lem:l11-residual-gap} yields
\[
  \|H(V\widehat x)-b_0\|_2\le\frac12.
\]
Thus the derived source solver succeeds with probability at least $0.99\cdot2/3>3/5$ and uses at most $2T$ matvecs.  By \cref{thm:solve-lb},
\[
  2T\ge c\frac{N}{\log^2(N+2)\,\log\log(N+e^e)}.
\]
Finally, with $d=2N$ and $\kappa=CN$,
\[
  \frac1{2\Lambda}
  =\frac{1}{2\sqrt2 C_{\rm emb}C}
    \frac{1}{N^{3/2}\sqrt{\log(2N+2)}}.
\]
Since $\log(2N+2)\le4\log(N+2)$, part (b) holds, for example, with $c_1=1/(4\sqrt2 C_{\rm emb}C)$ and with the same $c_2$ as in part (a).
\end{proof}

\begin{proof}[Proof of \cref{cor:rates}]
Both hard accuracy sequences are eventually strictly decreasing and satisfy $\eps_{N+1}/\eps_N\to1$.  Given sufficiently small $\eps$, choose $N$ so that
\[
  \eps_{N+1}<\eps\le\eps_N.
\]
An algorithm achieving gap at most $\eps$ would also achieve the weaker target $\eps_N$, so monotonicity of worst-case query complexity transfers the lower bound at $\eps_N$ to $\eps$.  The consecutive-ratio property ensures that this selected $N$ has the same asymptotic inversion as the equality $\eps=\eps_N$.

For the $\ell_2$--$\ell_1$ sequence, $\eps_N^{(2,1)}=\Theta(N^{-3/2})$, hence $N=\Theta(\eps^{-2/3})$ and the source lower bound gives \eqref{eq:l21-rate}.  For the $\ell_1$--$\ell_1$ sequence,
\[
  \eps_N^{(1,1)}
  =\Theta\!\left(N^{-3/2}(\log N)^{-1/2}\right),
  \qquad
  \eps^{-2/3}=\Theta\!\left(N(\log N)^{1/3}\right),
\]
which gives \eqref{eq:l11-rate}.  In both cases $\log N=\Theta(\log(1/\eps))$.  The displayed dimensions are polynomial in $1/\eps$, and \cref{lem:dimension-extension} extends each lower bound to all larger dimensions of the same geometry.
\end{proof}

\section{Scope and Remaining Questions}\label{sec:scope}
The lower bounds concern randomized worst-case query complexity with exact two-sided matrix-vector access.  For every randomized algorithm below the stated query threshold, the argument forces the existence of a fixed admissible matrix on which its success probability is less than $2/3$.  The instance may depend on the algorithm, but not on its realized private seed.  The auxiliary Haar rotation in the simplex--simplex embedding is used inside a reduction: a solver that succeeded on every admissible game would yield a source solver that succeeds on every promised linear system.  This averaging argument does not construct a universal Yao distribution over games.

The dimension regime is essential.  Any $m\times n$ matrix can be reconstructed using at most $\min\{m,n\}$ two-sided basis queries, after which the equilibrium can be computed without further oracle access.  Thus the lower bounds do not hold uniformly in fixed small dimensions.  \Cref{thm:main-sequences} identifies the hard sizes, and \Cref{lem:dimension-extension} extends the result by exact simulations.  Determining the full joint dependence on $(m,n,\eps)$ remains beyond these constructions.

All charged work is visible in the reductions.  One symmetric-dilation matvec costs one two-sided Gaussian query; one game query costs two ordinary source matvecs; and verification costs one source matvec per candidate.  The verification product is reused in refinement.  There are $O(\log N)$ inverse-power steps and $O(\log N)$ refinement stages per step, while confidence amplification contributes $O(\log\log N)$.  These account for the source denominator $\log^2N\log\log N$.  The signed dictionary adds the $\sqrt{\log N}$ accuracy loss responsible for the additional $\log^{1/3}(1/\eps)$ factor in the simplex--simplex lower bound.  Closing these logarithmic gaps requires sharper reductions.

The oracle model permits unlimited exact arithmetic on returned data.  The results therefore concern information obtained from matvecs, without asserting corresponding lower bounds on runtime, memory, communication, or parallel depth.  The Gaussian posterior and core extraction use exact conditional laws, so the proof does not infer hidden information merely from small numerical coefficients.  Its central unresolved verification task is the adaptive core argument; the probabilistic input estimates are established results with cited sources.  Subject to proof verification, the reductions establish the necessity of the $2/3$ exponent for randomized algorithms.


\clearpage
\begin{thebibliography}{99}

\bibitem{KOS26}
I.~Karmarkar, L.~O'Carroll, and A.~Sidford.
\newblock Solving matrix games with near-optimal matvec complexity.
\newblock arXiv:2601.02347v3, June 2026.
\newblock \url{https://arxiv.org/abs/2601.02347v3}.

\bibitem{KOS25}
I.~Karmarkar, L.~O'Carroll, and A.~Sidford.
\newblock Solving zero-sum games with fewer matrix-vector products.
\newblock In \emph{66th Annual IEEE Symposium on Foundations of Computer Science (FOCS)}, 2025.
\newblock Full version: arXiv:2509.04426v2, June 2026.
\newblock \url{https://arxiv.org/abs/2509.04426v2}.

\bibitem{KS26}
G.~Kornowski and O.~Shamir.
\newblock The oracle complexity of simplex-based matrix games.
\newblock arXiv:2412.06990v4, February 2026.
\newblock \url{https://arxiv.org/abs/2412.06990v4}.

\bibitem{BHSW20}
M.~Braverman, E.~Hazan, M.~Simchowitz, and B.~Woodworth.
\newblock The gradient complexity of linear regression.
\newblock In \emph{Proceedings of the Thirty Third Conference on Learning Theory}, volume 125 of \emph{Proceedings of Machine Learning Research}, pages 627--647, 2020.
\newblock \url{https://proceedings.mlr.press/v125/braverman20a.html}.

\bibitem{Nemirovski04}
A.~Nemirovski.
\newblock Prox-method with rate of convergence $O(1/t)$ for variational inequalities with Lipschitz continuous monotone operators and smooth convex-concave saddle point problems.
\newblock \emph{SIAM Journal on Optimization}, 15(1):229--251, 2004.
\newblock \url{https://doi.org/10.1137/S1052623403425629}.

\bibitem{Nesterov05}
Y.~Nesterov.
\newblock Smooth minimization of non-smooth functions.
\newblock \emph{Mathematical Programming}, 103:127--152, 2005.
\newblock \url{https://doi.org/10.1007/s10107-004-0552-5}.

\bibitem{HuangTikhomirov}
H.~Huang and K.~Tikhomirov.
\newblock A remark on the smallest singular value of powers of Gaussian matrices.
\newblock arXiv:1910.03702v3, January 2020.
\newblock \url{https://arxiv.org/abs/1910.03702v3}.

\bibitem{Vershynin}
R.~Vershynin.
\newblock \emph{High-Dimensional Probability: An Introduction with Applications in Data Science}.
\newblock Cambridge University Press, 2018.

\end{thebibliography}
\end{document}